\documentclass[11pt]{amsart}
\usepackage{setspace}
\usepackage[T1]{fontenc}
\numberwithin{equation}{section}
\numberwithin{table}{section}
\usepackage{amssymb, amsmath}
\usepackage{mathtools}
\usepackage{epsfig}
\usepackage{nicefrac}
\usepackage{graphicx}
\usepackage[margin=1.25in]{geometry}
\usepackage{tikz-cd}
\usepackage{lipsum}
\usepackage{amscd}
\input xy
\xyoption{all}

\theoremstyle{plain}
\newtheorem{thm}{Theorem}[section]

\newtheorem{lem}[thm]{Lemma}
\newtheorem{coro}[thm]{Corollary}
\newtheorem{prop}[thm]{Proposition}

\theoremstyle{definition} \theoremstyle{definition}

\theoremstyle{remark}

\def\N{{\mathbb N}}

\begin{document}
\title{On the Localization Map in the Galois Cohomology of Algebraic Groups}
\author{Dylon Chow}

\maketitle

\begin{abstract}
    We study surjectivity of a localization map in Galois cohomology.
\end{abstract}

\section{Introduction}

Let $F$ be a number field, and let $\Omega_F$ denote the set of places of $F$. If $v \in \Omega_F$, we denote by $F_v$ the completion of $F$ for the topology defined by $v$. Let $G$ be a connected reductive algebraic group defined over $F$. Each extension $F \rightarrow F_v$ defines a map $\omega_v:H^1(F,G) \rightarrow H^1(F_v,G)$. We thus get a canonical map \[\omega:H^1(F,G) \rightarrow \prod_{v \in \Omega_F} H^1(F_v,G).\] This map is a much-studied map. For instance, it is known that the kernel of $\omega$ is finite. If $\omega$ is injective, $G$ is said to satisfy the Hasse principle. Some algebraic groups, including simply connected groups and groups with trivial center, satisfy the Hasse principle, but some don't. 

The map $\omega$ is not always surjective, either. In fact, for tori, the failure of $\omega$ to be surjective is related to the failure of weak approximation for reductive groups. For $G=\rm{PGL}_n$, the Galois cohomology set $H^1(F,G)$ classifies equivalence classes of central simple algebras of degree $n^2$ over $F$. Every central simple algebra over $F$ splits almost everywhere, so to get a reasonable surjectivity result one must restrict to the direct sum of the $H^1(F_v,G)$. The map from $H^1(F,G)$ into the direct sum is not surjective in general, either. However, choose any non-archimedean place $v_0$ of $F$ and consider the diagonal map \[\varphi:H^1(F,G) \to \bigoplus_{v \neq v_0} H^1(F_v,G).\] This article is concerned with the following question: \textit{For which linear algebraic groups $G$ is $\varphi$ surjective?}

There are two papers that address this question. Borel and Harder (\cite{BorelHarder}) proved that $\varphi$ is surjective whenever $G$ is semisimple. They used this result to prove the existence of discrete cocompact subgroups in the groups of rational points of reductive groups over non-archimedean local fields of characteristic zero. Prasad and Rapinchuk (\cite{PRASAD2006646}) strengthened Borel and Harder's results and showed that $\varphi$ is surjective for some reductive groups that are not semisimple. This has applications to lattice counting problems (cf. \cite{belolipetsky2019}).

This article gives criteria for $\varphi$ to be surjective. One criterion involves the radical of $G$ and is essentially given in \cite{PRASAD2006646}. The other involves the cohomology of maximal tori in quasisplit groups. To obtain this criterion, we relate $\varphi$ to another local-global map involving the hypercohomology groups of a complex of tori, i.e., the "abelian cohomology groups" studied by Borovoi.

\section{Notation and Conventions}

\subsection{\ } We use $F$ to denote a field. Let $\overline{F}$ be an algebraic closure of $F$ and write $F^s$ for the separable closure of $F$ in $\overline{F}$. We let $\Gamma=\Gamma_F$ denote the Galois group of $F^s$ over $F$; it is a profinite topological group with the Krull topology. Then $H^i(F,H)$ denotes the $i$-th cohomology set of the Galois group $\Gamma$, with coefficients in $H(F_s)$ $(i=0,1)$ and, if $G$ is commutative, the $i$-th cohomology group of $\Gamma$ in $G(F_s)$ for all $i \in \N$.

\subsection{\ } If $F$ is a number field and $v$ is a place of $F$, then $F_v$ denotes the completion of $F$ at $v$.

\subsection{\ } In this section we review the construction of the abelian Galois cohomology groups. Let $F$ be a field of characteristic 0. For a connected reductive group $G$ over $K$, let $G'$ denote the derived group of $G$ and let $G^{sc}$ denote the universal covering group of $G'$. We consider the composition \[\rho:G^{sc} \rightarrow G' \hookrightarrow G.\] The complex $[G^{sc} \rightarrow G]$ is a crossed module (\cite{Borovoi1998} Lemma 3.7.1).

\subsection{\ }  We will use results on crossed modules. For the definition and basic properties of crossed modules, see \cite{Borovoi1998}. A morphism $f:T \rightarrow U$ of tori is defined over $F$ is a crossed module in a natural way and we can consider its cohomology groups $H^i(F,T \rightarrow U)$. We refer the reader to \cite{Borovoi1998} for the definition and properties of crossed modules and their cohomology. At the same time, $T \rightarrow U$ is also a complex of tori of length 2 and we can consider the Galois hypercohomology of this complex. In order to do this, we need to specify in what degrees the terms of this complex are placed. If we place $T$ in degree $-1$ and $U$ in degree $0$, then the Galois hypercohomology of the resulting complex coincides with the cohomology of the crossed module $T \rightarrow U$. This will be our convention, but we warn the reader that this convention is different from the one used in \cite[Appendix A]{KS99}, where the complexes of tori of length 2 are placed in degrees 0 and 1. Thus $H^1$ in our notation coincides with $H^2$ in the notation of \cite[Appendix A]{KS99}.

\subsection{\ } Let $T \subset G$ be a maximal torus defined over $F$ and let $T_{sc}$ be the inverse image $\rho^{-1}(T)$ of $T$ under $\rho$. We consider the complex of tori $T^\bullet = [T_{sc} \rightarrow T]$ where $T_{sc}$ is in degree $-1$ and $T$ is in degree $0$. We get a complex $T_{sc}(\overline{F}) \rightarrow T(\overline{F})$ of $\text{Gal}(\overline{F}/F)$-modules. 

\subsection{ \ } The natural morphism of crossed modules from $[1 \rightarrow G]$ to $[G^{sc} \rightarrow G]$ (\cite[section 3.10]{Borovoi1998}) induces a map in crossed module hypercohomology: \[\mathbb{H}^1(F,[1 \rightarrow G]) \rightarrow \mathbb{H}^1(F,[G^{sc} \rightarrow G]).\] We may identify $\mathbb{H}^1(F,[1 \rightarrow G])$ with $H^1(F,G)$. 

\subsection{\ } If $M$ and $N$ are commutative $\Gamma$-modules, then any homomorphism $M \rightarrow N$ of $\Gamma$-modules can be regarded as a crossed module with $N$ acting trivially on $M$. In this case $\mathbb{H}^i(F,M \rightarrow N)$ is the usual hypercohomology of the complex $M \rightarrow N$. This is the case when $T$ is a maximal $F$-torus in $G$ and $T_{sc}$ is the inverse image of $T$ under $\rho$. The $i$-th abelian Galois cohomology of $G$ is defined to be the $i$-th hypercohomology of the resulting complex: $H^i_{ab}(F,G):=\mathbb{H}^i(F,T^\bullet)$. These groups do not depend on the choice of $T$.

\subsection{\ } We recall the definition of the abelianization map. For any maximal torus $T$ in the reductive group $G$, the morphism of crossed modules \[[T_{sc}(\overline{F}) \rightarrow T(\overline{F})] \rightarrow [G^{sc}(\overline{F}) \rightarrow G(\overline{F})]\] is a quasi-isomorphism, and so it defines an isomorphism

\[\mathbb{H}^i(F,G^{sc} \rightarrow G) \rightarrow \mathbb{H}^i(F,T_{sc} \rightarrow T) = H^i_{ab}(F,G).\] In particular, we see that these sets are abelian groups. By composing, we get an abelianization map \[\text{ab}_1:H^1(F,G) \rightarrow H^1_{ab}(F,G).\] It is a morphism of pointed sets. In addition, we have an exact sequence (\cite{Borovoi1998} (3.10.1))
\[H^1(F,G^{sc}) \rightarrow H^1(F,G) \rightarrow H^1_{ab}(F,G).\]

\section{Surjectivity Results}

Our first criterion is in terms of the radical $R_G$ of $G$. We recall that if $G$ is a connected reductive group, $R_G$ is the connected component of the identity of the center of $G$; it is a torus in $G$. The quotient $G^{der}=G/R_G$ is semisimple.

\begin{prop}
Suppose that the following conditions are satisfied:

\begin{itemize}
    \item the map $\pi:H^1(F, R_G) \to \bigoplus_{v \neq v_0}H^1(F,R_G)$ is surjective, and
    
    \item the map $\rho:H^2(F,R_G) \to \bigoplus_{v \neq v_0}H^2(F,R_G)$ is injective.
\end{itemize}
Then \[\varphi:H^1(F,G) \to \bigoplus_{v \neq v_0}H^1(F_v,G)\] is surjective.
\end{prop}

\begin{proof}
The short exact sequence of algebraic groups \[1 \to R_G \to G \to G^{der} \to 1\] gives rise to a commutative diagram with exact rows:

\begin{tikzcd}
H^1(F,R_G) \arrow{r}{\alpha}\arrow{d}{\pi}
&H^1(F,G)\arrow{r}{\beta}\arrow{d}{\varphi}
&H^1(F,G^{der})\arrow{r}{\gamma}\arrow{d}{\sigma}
&H^2(F,R_G)\arrow{d}{\rho}\\
\bigoplus_{v \ne v_0}H^1(F_v,R_G)\arrow{r}{\alpha'}&\bigoplus_{v \ne v_0}H^1(F_v,G)\arrow{r}{\beta'}&\bigoplus_{v \ne v_0}H^1(F_v,G^{der})\arrow{r}{\gamma'}&\bigoplus_{v \neq v_0}H^2(F_v,R_G).
\end{tikzcd}

Let $y \in \bigoplus_{v \neq v_0}H^1(F_v,G)$. Suppose first that $\beta'(y)=0$. Then there is $x \in \bigoplus_{v \neq v_0}H^1(F_v,R_G)$ such that $\alpha'(x)=y$. Since we assumed that $\pi$ is surjective, there is some $w \in H^1(F,R_G)$ such that $\pi(w) = x$, and $\varphi(\alpha(x))=y$.

Now we consider the general case. Since $\sigma$ is surjective, there is $b \in H^1(F,G^{der})$ such that $\sigma(b)=\beta'(y)$. Since $\gamma'(\beta'(y))=0$, it follows that $\rho(\gamma(b))=0$. Since $\rho$ is assumed to be injective, $\gamma(b)=0$, and so we can find $a \in H^1(F,G)$ such that $\beta(a)=b$. If we twist $G$ by a cocycle belonging to $a$, then we are reduced to the case where $\beta'(y)=0$. This completes the proof.
\end{proof}

Now we will give a more flexible criterion in terms of maximal tori. First, we reformulate our conjecture in terms of the Galois hypercohomology of complexes of tori. Let $G^{sc}$ be the simply connected cover of $G^{der}$ and let $\rho$ be the composition \[\rho:G^{sc} \to G^{der} \hookrightarrow G.\] Let $T$ be a maximal torus in $G$ and let $T_{sc}=\rho^{-1}(T)$. The complex of tori $G_{ab}:=[T_{sc} \to T]$ forms a crossed module, and its first hypercohomology group ($T_{sc}$ is in degree $-1$) is called the first abelian cohomology group of $G$: $H^1_{ab}(F,G)=\mathbb{H}^1(F,[T_{sc} \to T])$. There is an "abelianization map" $\text{ab}_1:H^1(F,G) \to H^1_{ab}(F,G)$.

\begin{thm}
Let $G$ be a reductive group over $F$. Let $v_0$ be any non-archimedean place of $F$. The map $\varphi$ is surjective if and only if the map \[H^1_{ab}(F,G) \to \bigoplus_{v \neq v_0} H^1_{ab}(F_v,G)\] is surjective.
\end{thm}

\begin{proof}
We have the commutative diagram with exact rows

\begin{tikzcd}
H^1(F,G^{sc}) \arrow{r}{\mu}\arrow{d}{\omega''}
&H^1(F,G)\arrow{r}{\text{ab}_1}\arrow{d}{\omega}
&H^1_{ab}(F,G)\arrow{d}{\omega'}\\
\bigoplus_{v \ne v_0}H^1(F_v,G^{sc})\arrow{r}{\mu'}&\bigoplus_{v \ne v_0}H^1(F_v,G)\arrow{r}{\text{ab}_1'}&\bigoplus_{v \ne v_0}H^1_{ab}(F_v,G).
\end{tikzcd}

Suppose that $\omega:H^1(F,G) \to \bigoplus_{v \neq v_0} H^1(F_v,G)$ is surjective. Since $\text{ab}_1'$ is surjective (cf. \cite[Theorem 5.4]{Borovoi1998}), $\text{ab}_1' \circ \omega$ is surjective, and so $\omega'$ is also surjective.

Conversely, suppose that $\omega':H^1_{ab}(F,G) \to \bigoplus_{v \neq v_0} H^1_{ab}(F_v,G)$ is surjective and let $a \in \bigoplus_{v \neq v_0} H^1(F_v,G)$. We want to find some $x \in H^1(F,G)$ such that $\omega(x)=a$. If $a$ is in the image of  $\mu'$, then our assertion follows from the surjectivity of $\omega''$ (\cite{BorelHarder}, p.57 for details). 

For the general case, we use a twisting argument. By assumption, $\omega'$ is surjective. On the other hand, $\text{ab}_1$ is also surjective (\cite[Theorem 5.7]{Borovoi1998}). Therefore, we may find an element $b \in H^1(F,G)$ such that $\omega'(\text{ab}_1(b))=\text{ab}_1^{'}(a)$. If we twist $G$ by any cocycle belonging to $b$, then by a standard twisting argument adapted to the case of hypercohomology \cite[Proposition 3.16]{Borovoi1998}, we are reduced to the case where $\text{ab}_1'(a)=0$. From this our assertion follows.
\end{proof}

Up to quasi-isomorphism, the crossed module $G_{ab}=[T_{sc} \to T]$ is insensitive to inner automorphisms (\cite[Corollary 2.9]{Borovoi1998}), so we will assume henceforth that $G$ is quasisplit. To state our next criterion, we use the exact sequence \cite[A.1.1]{KS99} or \cite[Proposition 2.12]{Borovoi1998} \[\dots \to H^1(F,T) \to H^1_{ab}(F,G) \to H^2(F,T_{sc}) \to H^2(F,T) \to \dots,\] which gives a commutative diagram with exact rows 

\begin{tikzcd}
H^1(F,T) \arrow{r}\arrow{d}
&H^1_{ab}(F,G)\arrow{r}\arrow{d}
&H^2(F,T_{sc})\arrow{r}\arrow{d}
&H^2(F,T)\arrow{d}\\
\bigoplus_{v \ne v_0}H^1(F_v,T)\arrow{r}&\bigoplus_{v \ne v_0}H^1_{ab}(F_v,G)\arrow{r}&\bigoplus_{v \ne v_0}H^2(F_v,T_{sc})\arrow{r}&\bigoplus_{v \neq v_0}H^2(F_v,T).
\end{tikzcd}

\begin{lem}
Let $G$ be a simply connected algebraic group. There is a maximal torus $T$ in $G$ such that the map $H^2(F,T) \to \bigoplus_{v \neq v_0}H^2(F_v,T)$ is surjective.
\end{lem}

\begin{proof}
Since $G$ is simply connected, it contains an induced torus $T$, i.e., $T$ is a finite product of the form $R_{E/F}(\rm{GL}_1)$ with $E$ a finite extension of $F$. By class field theory, the claim holds for $T=R_{E/F}(\rm{GL}_1)$. 
\end{proof}

By considering the Five Lemma, we deduce the following:

\begin{prop}
Let $G$ be a connected quasisplit group defined over a number field $F$. Let $v_0$ be a non-archimedean place of $F$. Suppose there exists a maximal torus $T$ in $G$ such that the following hold:

\begin{itemize}

\item $T_{sc}$ is induced;

\item the map $H^1(F,T) \to \bigoplus_{v \neq v_0} H^1(F_v,T)$ is surjective; and

\item the map $H^2(F,T) \to \bigoplus_{v \neq v_0}H^2(F_v,T)$ is injective.
\end{itemize} 

Then the map $\varphi:H^1(F,G) \to \bigoplus_{v \neq v_0}H^1(F_v,G)$ is surjective. Moreover, for each inner form $G'$ of $G, H^1(F,G') \to \bigoplus_{v \neq v_0}H^1(F_v,G')$ is also surjective.
\end{prop}

\begin{coro}
If $G$ is split, or more generally an inner form of a split group, then $\varphi:H^1(F,G) \to \bigoplus_{v \neq v_0} H^1(F_v,G)$ is surjective.
\end{coro}

\section*{Acknowledgements}

The author thanks Mikhail Borovoi for helpful comments and Ramin Takloo-Bighash for his interest in this work.

\bibliographystyle{alpha}
\bibliography{biblio1.bib}

\end{document}